\documentclass[12pt,letterpaper]{amsart}
\usepackage{geometry}
\usepackage{hyperref}
\usepackage{mathpazo}
\usepackage{amssymb,url,setspace,xcolor}
\newtheorem{theorem}{Theorem}
\newtheorem*{theorem*}{Theorem}

\newtheorem{lemma}{Lemma}
\newtheorem{corollary}{Corollary}
\theoremstyle{remark}
\newtheorem{remark}{Remark}
\newcommand{\C}{\mathbb{C}}
\newcommand{\ep}{\varepsilon}
\newcommand{\D}{\Omega}
\newcommand{\Dc}{\overline{\Omega}}
\newcommand{\dbar}{\overline{\partial}}
\newcommand{\zb}{\overline{z}}

\onehalfspace

\title[Compactness of Hankel operators with continuous symbols]{Compactness 
	of Hankel operators with symbols continuous on the closure of pseudoconvex 
	domains}

\author{Timothy G. Clos}
\address[Timothy G. Clos]{Bowling Green State University, 
	Department of Mathematics and Statistics,  Bowling Green, Ohio 43403 }
\curraddr{University of Toledo, Department of 
	Mathematics \& Statistics, 2801 W. Bancroft, Toledo, OH 43606}
\email{clost@bgsu.edu}

\author{Mehmet \c{C}el\.ik}
\address[Mehmet \c{C}elik]{Texas A\&M University - Commerce, 
	Department of Mathematics, Commerce, TX, 75428}
\email{mehmet.celik@tamuc.edu}

\author{S\"{o}nmez \c{S}ahuto\u{g}lu}
\address[S\"{o}nmez \c{S}ahuto\u{g}lu]{University of Toledo, Department of 
	Mathematics \& Statistics, 2801 W. Bancroft, Toledo, OH 43606, USA}
\email{sonmez.sahutoglu@utoledo.edu}

\date{\today}
\begin{document}
\begin{abstract}
Let $\Omega$ be a bounded pseudoconvex domain in $\mathbb{C}^2$ 
with Lipschitz boundary or a bounded convex domain in $\mathbb{C}^n$ 
and $\phi\in C(\overline{\Omega})$ such that the Hankel operator 
$H_{\phi}$ is compact on the Bergman space $A^2(\Omega)$. 
Then $\phi\circ f$ is holomorphic for any holomorphic 
$f:\mathbb{D}\rightarrow b\Omega$.  
\end{abstract}
%%%%%%%%%%%%%%%%%%%%%%%%%
\maketitle

Let $\D$ be a domain in $\mathbb{C}^n$ and $A^2(\D)$ denote 
the Bergman space of $\D$, the space of square integrable 
holomorphic functions on $\D$.  Since $A^2(\D)$ is a closed 
subspace of $L^2(\D)$, the space of square integrable functions 
on $\D$, there exists an orthogonal projection 
$P:L^2(\D)\rightarrow A^2(\D)$, called the Bergman projection.  
The Hankel operator $H_{\phi}:A^2(\D)\to L^2(\D)$ with symbol 
$\phi\in L^{\infty}(\D)$ is defined as $H_{\phi}f=(I-P)(\phi f)$ where 
$I$ denotes the identity operator. Hankel operators have been 
well studied on the Bergman space of the  unit disc. 
Sheldon Axler in \cite{Axler86} proved the following 
interesting theorem.

\begin{theorem*}[Axler]
Let $\phi\in A^2(\mathbb{D})$. Then $H_{\overline{\phi}}$ is 
compact if and only if $(1-|z|^2)\phi'(z)\to 0$ as $|z|\to 1$. 
\end{theorem*}

The space of holomorphic functions satisfying the condition 
in the theorem is called little Bloch space. One can check 
that $\phi(z) = \exp((z+1)/(z-1))$ is bounded on $\mathbb{D}$ 
but it does not belong to the little Bloch space. Hence not every 
bounded symbol that is smooth on the domain produces compact 
Hankel  operator on the disc. However, Hankel operators with symbols 
continuous on the closure are compact for bounded domains in $\C$ 
(see, for instance, \cite[Proposition 1]{Sahutoglu12}).  We refer the 
reader to \cite{ZhuBook} for more  information on the theory of  
Hankel operators (as well as Toeplitz  operators) on the Bergman 
space of the unit disc. We note that Sheldon Axler's result has been 
extended to a small class of domains in $\C^n$, such as strongly 
pseudoconvex domains, by Marco Peloso \cite{Peloso94} and 
Huiping Li \cite{Li94}.

The situation in $\C^n$ for $n\geq 2$ is radically different. For instance, 
$H_{\zb_1}$ is not compact when $\D$ is the bidisc (see, for instance,  
\cite{Le10,ClosThesis,ClosSahutoglu18,Clos17}).  Hence in higher 
dimensions compactness of Hankel operators is not guaranteed 
even if the symbol is smooth up to the boundary. We refer the 
reader to \cite{StraubeBook,HaslingerBook} for more information 
about Hankel operators in higher dimensions and their relations 
to $\dbar$-Neumann problem.

We are interested in studying compactness of Hankel operators on 
Bergman spaces defined on domains in $\C^n$. We would like to 
understand  compactness of Hankel operators in terms of the interaction 
of the  symbol with the boundary geometry. This interaction does not 
surface for domains in $\C$ as the boundary has no complex geometry. 
However, to relate the symbol to the boundary geometry we will restrict 
ourselves to symbols that are at least continuous up to the boundary. 
The first results in this direction are due to \v{Z}eljko \v{C}u\v{c}kovi\'c 
and the third author in \cite{CuckovicSahutoglu09}. They obtain results  
about compactness of Hankel operators in terms of the behavior of the 
symbols along analytic discs in the boundary, on smooth bounded 
pseudoconvex domains (with a restriction on the Levi form) 
and on smooth bounded convex domains in $\C^n$. Moreover, 
for convex domains in $\C^2$ they obtain a characterization 
for compactness (see \cite[Corollary 2]{CuckovicSahutoglu09}).
We note that even though they state their results for $C^{\infty}$-smooth 
domains and symbols, observation of the proofs shows that only 
$C^1$-smoothness is sufficient. One of their results, stated  with $C^1$ 
regularity, is the following theorem. 
   
\begin{theorem*}[\v{C}u\v{c}kovi\'c-\c{S}ahuto\u{g}lu]
Let $\D$ be a $C^1$-smooth bounded convex domain in $\C^2$ 
and $\phi\in C^1(\Dc)$. Then $H_{\phi}$ is compact if and only if 
$\phi\circ f$ is holomorphic for any holomorphic $f:\mathbb{D}\to b\D$.	
\end{theorem*}

The theorem above can be interpreted as follows: $H_{\phi}$ is compact 
if and only if  $\phi$ is ``holomorphic along" every non-trivial analytic 
disc in the boundary. 

The situation for symbols that are only continuous up to the boundary 
is less understood. When $\D$ is a bounded convex domain in $\C^n$ 
with no non-trivial discs in $b\D$ (that is, any holomorphic mapping 
$f:\mathbb{D}\to b\D$ is constant) all of the Hankel operators with 
symbols continuous on $\Dc$  are compact. This follows from the 
following two facts: on such domains the $\dbar$-Neumann operator 
is compact (see \cite{FuStraube98}); compactness of the $\dbar$-Neumann 
operator implies that Hankel operators with symbols continuous 
on closure  are compact (see \cite[Proposition 4.1]{StraubeBook}). 

 In case of the polydisc Trieu Le in \cite{Le10} proved the 
following  characterization. 

\begin{theorem*}[Le]
Let $\phi$ be continuous on $\overline{\mathbb{D}^n}$ for $n\geq 2$. 
Then $H_{\phi}$ is compact if and only if there exist 
$\phi_1,\phi_2\in C(\overline{\mathbb{D}^n})$ such that  
$\phi_1$ is holomorphic on $\mathbb{D}^n,  \phi_2=0$ on 
$b\mathbb{D}^n$, and $\phi=\phi_1+\phi_2$.
\end{theorem*}

A domain $\D\subset \C^n$ is called Reinhardt if $(z_1,\ldots, z_n)\in \D$ 
implies that $(e^{i\theta_1}z_1,\ldots, e^{i\theta_n}z_n)\in \D$ for 
any $\theta_1,\ldots,\theta_n\in \mathbb{R}$. That is, Reinhardt domains 
are invariant under rotation in each variable. These are generalizations 
of the ball and the polydisc. Reinhardt domains are useful in describing 
domain of convergence for power series centered at the origin (see, 
for instance, \cite{KrantzBook,NarasimhanBook,RangeBook}).

Motivated by the previous results mentioned above, recently, the first 
and the last authors proved the following result on convex Reinhardt 
domains in $\C^2$ (see \cite{ClosSahutoglu18}), generalizing the results in 
\cite{CuckovicSahutoglu09} (in terms of regularity of the symbol but 
on a small class of domains) and \cite{Le10} (in terms of the domain 
in $\C^2$).

\begin{theorem*}[Clos-\c{S}ahuto\u{g}lu]
Let $\D$ be a bounded convex Reinhardt domain in $\C^2$ and 
$\phi\in C(\Dc)$. Then $H_{\phi}$ is compact if and only if 
$\phi\circ f$ is holomorphic for any holomorphic $f:\mathbb{D}\to b\D$.	
\end{theorem*}
 
We note that on piecewise smooth bounded convex Reinhardt domains 
in $\C^2$, the first author studied compactness of Hankel operators 
with conjugate holomorphic square integrable functions  in \cite{Clos17}. 
Furthermore, compactness of products of two Hankel operators with 
symbols continuous up to the boundary was studied by     
\v{Z}eljko \v{C}u\v{c}kovi\'c and the last author 
in \cite{CuckovicSahutoglu14}. 

In this paper we are able to partially generalize the result of 
Clos-\c{S}ahuto\u{g}lu to more general domains. In case the 
domain is in $\C^2$ we have the following result. 

\begin{theorem}\label{ThmC2}
Let $\D$ be a bounded pseudoconvex domain in $\mathbb{C}^2$ 
with Lipschitz boundary and $\phi\in C(\Dc)$ such that $H_{\phi}$ 
is compact on $A^2(\D)$. Then $\phi\circ f$ is holomorphic for any 
holomorphic $f:\mathbb{D}\rightarrow b\D$.  	
\end{theorem}

However, for convex domains we can prove the following 
result in $\C^n$.  
\begin{theorem}\label{ThmConvex}
Let $\D$ be a bounded convex domain in $\C^n$ and 
$\phi\in C(\Dc)$ such that $H_{\phi}$ is compact on 
$A^2(\D)$. Then $\phi\circ f$ is holomorphic for any 
holomorphic  $f:\mathbb{D}\rightarrow b\D$.  	
\end{theorem}

As a corollary of Theorem \ref{ThmConvex} we obtain the 
following result for locally convexifiable domains in $\C^n$.

\begin{corollary}\label{Cor1}
Let $\D$ be a bounded locally convexifiable domain in $\C^n$ and 
$\phi\in C(\Dc)$ such that $H_{\phi}$ is compact on $A^2(\D)$. 
Then $\phi\circ f$ is holomorphic for any holomorphic  
$f:\mathbb{D}\rightarrow b\D$.  
\end{corollary}

A domain $\D\subset \C^n$ is called complete Reinhardt if 
$(z_1,\ldots, z_n)\in \D$ and  $\xi_1,\ldots,\xi_n\in \C$ with 
$|\xi_j|\leq 1$ for all $j$ then  $(\xi_1z_1,\ldots, \xi_nz_n)\in \D$. 
We note that convex Reinhardt domains are complete Reinhardt 
but the converse is not true. 

As a second corollary we obtain the following result for pseudoconvex 
complete Reinhardt domains in $\C^2$. 

\begin{corollary}\label{Cor2}
Let $\D$ be a bounded pseudoconvex complete Reinhardt 
domain in $\C^2$ and $\phi\in C(\Dc)$ such that $H_{\phi}$ 
is compact on $A^2(\D)$. Then $\phi\circ f$ is holomorphic 
for any holomorphic  $f:\mathbb{D}\rightarrow b\D$. 
\end{corollary}

\begin{remark}
Peter Matheos, in his thesis \cite{MatheosThesis} (see also 
\cite[Theorem 10]{FuStraube01} and \cite[Theorem 4.25]{StraubeBook}), 
constructed a smooth bounded pseudoconvex complete Hartogs 
domain in $\C^2$ that has no analytic disc in its boundary, yet the 
$\dbar$-Neumann operator on the domain is not compact. Furthermore, 
Zeytuncu and the third author \cite[Theorem 1]{SahutogluZeytuncu17} 
proved that on smooth bounded pseudoconvex Hartogs domains in $\C^2$, 
compactness of the $\dbar$-Neumann operator is equivalent to compactness 
of all Hankel operators with symbols smooth up to the boundary. 
Therefore, on Matheos' example the condition of Theorem \ref{ThmC2} 
is trivially satisfied, yet there exists a non-compact Hankel operator 
with a symbol smooth on the closure of the domain. Namely, 
the converse of Theorem \ref{ThmC2} is not true. On the other 
hand, the converse of Theorem \ref{ThmConvex} is open. 	    
\end{remark}

The plan of the paper is as follows: First we will prove a localization 
result for compactness of Hankel operators with bounded (not 
necessarily continuous) symbols. Then we concentrate on the proof of  
Theorem \ref{ThmC2}. Finally, we prove Theorem \ref{ThmConvex} 
and the corollaries. 

%%%%%%%%%%%%%%%%%%%%%%%%%%%%%%%%%%%%%%
%%%%%%%%%%%%%%%%%%%%%%%%%%%%%%%%%%%%%%
\section*{Localization of compactness}

We note that $H^U_{\phi}$ denotes  the Hankel operator on $A^2(U)$ 
with symbol $\phi$ which is an essentially bounded function on a domain $U$. 
Furthermore, we will use the following notation: $A\lesssim B$ means 
that there exists $c>0$ that does not depend on quantities of interest 
such that  $A\leq cB$. Also the constant $c$ might change at every 
appearance. In the following lemma and the rest of the paper, 
$B(p,r)$ denotes the open ball centered at $p$ with radius $r$.

%%%%%%%%%%%%%%%%%%%%%%%%%%%%%%%%%%%%%
\begin{lemma}\label{LemLocalization}
Let  $\D$ be a  bounded pseudoconvex domain in $\C^n$, 
$\phi \in L^{\infty}(\D)$, $p\in  b\D, 0<r_1<r_2$, and 
$R_{r_2,r_1}: A^2(B(p,r_2)\cap \D)\to A^2(B(p,r_1)\cap \D)$ be the 
restriction operator defined as $R_{r_2,r_1}f=f|_{B(p,r_1)\cap \D}$. 
Assume that  $H^{\D}_{\phi}$ is compact on $A^2(\D)$. Then 
$H^{B(p,r_1)\cap \D}_{\phi}R_{r_2,r_1}$ is compact on 
$A^2(B(p,r_2)\cap \D)$.
\end{lemma}
\begin{proof}
First we will simplify the notation and define the necessary operators. 
Let $U_j=B(p,r_j)\cap \D, Q^{U_j}=I-P^{U_j}:L^2(U_j)\to L^2(U_j)$ 
for $j=1,2$, and  $Q^{\D}=I-P^{\D}:L^2(\D)\to L^2(\D)$. 
Also, in the following calculations $\|.\|_{U_j}$ and $\|.\|_{\D}$ 
denote the $L^2$ norms on $U_j$ and $\D$, respectively. 

By \cite[Lemma 3]{Sahutoglu12} we have a bounded operator 
$E_{\ep}:A^2(U_2)\to A^2(\D)$ with the  following estimate 
\[\| R_{U_1}(f-E_{\ep}f)\|_{U_1}\leq \ep\|R_{U_1}f\|_{U_1}\]
for $f\in A^2(U_2)$ where $R_{U_1}$ denotes the restriction onto $U_1$. 
Then,  $H^{U_1}_{\phi}R_{U_1}g=Q^{U_1}R_{U_1}H^{\D}_{\phi}g$ for any 
$g\in A^2(\D)$. Then for $f\in A^2(U_2)$ we have  
\begin{align*}
\left\|H^{U_1}_{\phi}R_{r_2,r_1}f \right\|_{U_1}^2 
=& \langle R_{U_1}(\phi (f-E_{\ep} f)),Q^{U_1}R_{U_1}(\phi f) \rangle_{U_1}
	+\langle R_{U_1}(\phi E_{\ep} f),Q^{U_1}R_{U_1}(\phi f) \rangle_{U_1}  \\
\lesssim & \ep \|R_{U_1}f\|_{U_1}^2 + \left|\langle R_{U_1}(\phi E_{\ep} f), 
	Q^{U_1}R_{U_1}(\phi (f-E_{\ep} f)) \rangle_{U_1}\right| 
	+ \left\| H^{U_1}_{\phi}R_{U_1}(E_{\ep}f)\right\|_{U_1}^2\\
\lesssim & (\ep+\ep(1+\ep))\|R_{U_1}f\|_{U_1}^2  
	+ \left\|Q^{U_1}R_{U_1}H^{\D}_{\phi}E_{\ep}f \right\|_{U_1}^2\\
\lesssim & \ep(2+\ep)\|R_{U_1}f\|_{U_1}^2
	+\left\|R_{U_1}H^{\D}_{\phi}E_{\ep}f \right\|_{U_1}^2.
\end{align*}
Next we will use the  compactness characterization of operators in 
\cite[Lemma 4.3]{StraubeBook} (see also 
\cite[Proposition V.2.3]{D'AngeloBook}). Since $H^{\D}_{\phi}$ is 
compact, for every $\ep'>0$ there exists a compact operator 
$K_{\ep'}:A^2(\D)\to L^2(\D)$ such that 
\[\left\|R_{U_1}H^{\D}_{\phi}E_{\ep}f \right\|_{U_1}^2
	\leq \left\|H^{\D}_{\phi}E_{\ep}f \right\|^2_{\D}\leq
	\ep'\|E_{\ep}f\|^2_{\D}+\|K_{\ep'}E_{\ep}f\|^2_{\D}.\] 
Therefore, we have 
\[\left\|R_{U_1}H^{\D}_{\phi}E_{\ep}f \right\|_{U_1}^2
	\leq \ep'\|E_{\ep}\|^2\|f\|_{U_1}^2 
	+\|K_{\ep'}E_{\ep}f\|^2_{\D}.\]
We note that $K_{\ep'}E_{\ep}:A^2(U_2)\to L^2(\D)$ is compact 
for any $\ep$ and $\ep'.$  Now we choose $\ep'$ sufficiently 
small so that $\ep'\|E_{\ep}\|^2<\ep.$ Hence, there exists 
$C>0$ (independent of $\ep, \ep'$ and $f$) such that 
\[\left\|H^{U_1}_{\phi} R_{r_2,r_1}f\right\|_{U_1}^2
	\leq  C\ep(3+\ep)\|f\|_{U_2}^2
	+\|K_{\ep'}E_{\ep}f\|^2_{\D}.\]
Finally, \cite[Lemma 4.3]{StraubeBook} 
implies that $H^{U_1}_{\phi}R_{r_2,r_1}$ is compact 
on $A^2(U_2).$
\end{proof}

\begin{remark}
We note that the third author proved a localization result previously 
in \cite{Sahutoglu12}. In \cite[Theorem 1]{Sahutoglu12} the domain 
may be very irregular but the symbol was assumed to be $C^1$-smooth 
up to the boundary. In Lemma \ref{LemLocalization}, however,  
we assume that the symbol is only bounded on the domain.
\end{remark}

%%%%%%%%%%%%%%%%%%%%%%%%%%%%%%%%%%%
%%%%%%%%%%%%%%%%%%%%%%%%%%%%%%%%%%%
\section*{Proof of Theorem \ref{ThmC2}}

%%%%%%%%%%%%%%%%%%%%%%%%%%%%%%%%%%
\begin{lemma}\label{LemConv}
Let $U$ be a domain in $\C$ and $\phi\in C(U)$ that is not 
holomorphic. Assume that  $\{\phi_k\}\subset C^1(U)$ such that 
$\langle\phi_k, h\rangle \to \langle\phi,h\rangle$ as $k\to\infty$ 
for all $h\in C^{\infty}_0(U)$. Then there exists a subsequence 
$\{\phi_{k_j}\},\delta>0,$ and $h\in C^{\infty}_0(U)$ such that 	
\[\left|\left\langle \frac{\partial \phi_{k_j}}{\partial \overline{z}}, 
	h \right\rangle\right|\geq \delta\]  
for all $j$.
\end{lemma}

\begin{proof}
We want to show that there exists $h\in C^{\infty}_0(U)$ such that 
$\langle (\phi_k)_{\zb}, h\rangle$ does not converge to $0$ 
as $k\rightarrow \infty$. Suppose that 
$\langle (\phi_k)_{\zb}, h\rangle\rightarrow 0$ as 
$k\rightarrow \infty$  for all $h\in C^{\infty}_0(U) $.  Then, 
\[\langle \phi_k,h_z \rangle\rightarrow \langle \phi, h_z\rangle 
\text{ as } k\rightarrow\infty.\] 
Hence, in the limit we have $\langle \phi,h_z\rangle=0$ for all 
$h\in C^{\infty}_0(U)$. This implies that $\phi$ is in the kernel 
of the $\overline{\partial}$ operator (in the distribution sense) 
on $U$. In particular, $\phi$ is harmonic. Then $\phi$ is 
$C^{\infty}$-smooth (see, for instance, \cite[Corollary 2.20]{FollandBook}) 
and, in turn, it is holomorphic. This contradicts with the assumption 
that  $\phi$ is not holomorphic. 

Therefore, there exists 
$\delta>0, h\in C^{\infty}_0(U)$, and a subsequence 
$\phi_{k_j}$ such that 
\[|\langle (\phi_{k_j})_{\zb},h\rangle|\geq \delta\] 
for all $j$.
\end{proof}

%%%%%%%%%%%%%%%%%%%%%%%%%%%%%%
\begin{lemma}\label{LemHoloInvariance}
Let $\D_1$ and $\D_2$ be two bounded domains in $\C^n$, $F:\D_1\to \D_2$ 
be a biholomorphism, and $\phi\in L^{\infty}(\D_2)$. Furthermore, let $U_1$ 
is an open set in $\D_1, F(U_1)=U_2$, and $R_j:A^2(\D_j)\to A^2(U_j)$ be the 
restriction operators for $j=1,2.$ Assume that $H_{\phi}^{U_2}R_2$ is compact. 
Then  $H_{\phi\circ F}^{U_1}R_1$ is compact. 
\end{lemma}
\begin{proof}
First,  we mention the following formula about Bergman projections.  
Let $J_F$ denote the determinant of  (complex) Jacobian of $F$ and 
$g\in L^2(U_2)$. Then by \cite[Theorem 1]{Bell81} we have 
\begin{align*}
P_{U_1}(J_F\cdot (g\circ F)) =J_F\cdot P_{U_2}(g)\circ F
\end{align*} 
(see also \cite[Proof of Theorem 12.1.11]{JarnickiPflugBook}). 

Next, using the equality above, we will get an equality between 
Hankel operators on $U_1$ and $U_2$. To that end let $h$ be a 
square integrable holomorphic function on $U_2$. Then 
\begin{align*}
H_{\phi\circ F}^{U_1}(h\circ F)
=& \phi\circ F\cdot  h\circ F- P_{U_1}(\phi\circ F\cdot  h\circ F)\\
=&J_F\cdot \phi \circ F\cdot \frac{h\circ F}{J_F}-J_F\cdot P_{U_2}
	\left( \frac{\phi\circ F\cdot  h\circ F}{J_F}\circ F^{-1} \right)\circ F\\
=& J_F \cdot \left(\phi \cdot \frac{h}{J_F\circ F^{-1}} -P_{U_2} 
	\left(\phi \cdot \frac{h}{J_F\circ F^{-1}}\right) \right)\circ F \\
=& J_F \cdot  H_{\phi}^{U_2}\left(\frac{h}{J_F\circ F^{-1}}\right)\circ F.
\end{align*}

We need to make sure that  
$\frac{f \circ F^{-1}}{J_F\circ F^{-1}}\in A^2(\D_2)$ 
for any $f\in A^2(\D_1)$. In fact,  
\begin{align*} 
\left\|\frac{f\circ F^{-1}}{J_F\circ F^{-1}}\right\|_{\D_2}^2
=& \int_{\D_2} \left|\frac{f \circ F^{-1}(w)}{J_F\circ F^{-1}(w)}
	\right|^2dV(w)\\
 =& \int_{\D_1} \left|\frac{f \circ F^{-1}(F(z))}{J_F\circ F^{-1}(F(z))}
	\right|^2|J_F(z)|^2dV(z)\\
=& \int_{\D_1} \left|\frac{f(z)}{J_F(z)}\right|^2|J_F(z)|^2dV(z)\\
=& \int_{\D_1} |f(z)|^2dV(z)\\
 =&\left\|f\right\|_{\D_1}^2.
\end{align*}

So far we have shown that if $\{f_j\}$ is a bounded sequence in $A^2(\D_1)$ 
then $\left\{f_j\circ F^{-1}/(J_F\circ F^{-1})\right\}$ is a bounded 
sequence in $A^2(\D_2)$ and 
\begin{align}\label{EqnBell}
H_{\phi\circ F}^{U_1}R_1(f_j)
	=J_F \cdot H_{\phi}^{U_2}R_2 
	\left(\frac{f_j \circ F^{-1}}{J_F\circ F^{-1}}\right)\circ F.
\end{align} 
Then compactness of $H_{\phi}^{U_2}R_2$ implies that  
$\left\{H_{\phi}^{U_2}R_2
\left(\frac{f_j \circ F^{-1}}{J_F\circ F^{-1}}\right)\right\}$ has a 
convergent subsequence in $L^2(U_2)$. Using the fact that 
$\|h\|_{U_2}=\|J_F\cdot h\circ F\|_{U_1}$ for any $h\in L^2(U_2)$ 
together with \eqref{EqnBell}, we conclude that 
$\{H_{\phi\circ F}^{U_1}R_1(f_j)\}$ has a convergent 
subsequence in $L^2(U_1)$. 
Therefore, $H_{\phi\circ F}^{U_1}R_1$ is compact. 
\end{proof}

Let $\chi \in C^{\infty}_0(B(0,1))$ such that $\int_{B(0,1)}\chi(z)dV(z)=1$. 
We  define  
\[\chi_k(z)=k^{2n}\chi(kz)\]
for $k=1,2,3,\ldots$

%%%%%%%%%%%%%%%%%%%%%%%%%%%%%%%%%%%%%%%%
%%%%%%%%%%%%%%%%%%%%%%%%%%%%%%%%%%%%%%%%
\begin{proof}[Proof of Theorem \ref{ThmC2}]
We assume that $H_{\phi}$ is compact and there is a holomorphic 
map $f:\mathbb{D}\to b\D$ such that $\phi\circ f$ is not holomorphic. 
Then $f$ is a non-constant mapping. We can use Lemma \ref{LemLocalization} 
to localize the compactness of $H_{\phi}$ near a boundary point 
$f(\xi_0)=p\in b\D$ such that $\phi\circ f$ is not holomorphic near 
$\xi_0$. That is, we choose $0<r_1<r_2$ such that 
$H^{\D\cap B(p,r_1)}_{\phi}R_{r_2,r_1}$ is compact on 
$A^2(\D\cap B(p,r_2))$ and $\phi\circ f$ is not holomorphic 
on $f^{-1}(B(p,r_1))$. To simplify the geometry 
we want to straighten the disc near $p$ yet keep compactness of the 
Hankel operator locally. So, shrinking $r_1,r_2$ if necessary, we 
use a local holomorphic change of coordinates 
\[F:\D\cap B(p,r_2)\to \C^2\] 
so that $F\circ f$ maps $f^{-1}(B(p,r_2))$ onto 
an open set on $z_1$-axis and $F\circ f(\xi_0)=0$. 

To simplify the notation, let us denote $\D_1=F(\D\cap B(p,r_1))$ and 
$\D_2=F(\D\cap B(p,r_2))$. Lemma \ref{LemHoloInvariance} 
implies that $H_{\phi \circ F^{-1}}^{\D_1}R$ is compact on 
 $A^2(\D_2)$ where $R:A^2(\D_2)\to A^2(\D_1)$ is the 
 restriction operator. Therefore, without loss of  generality,  
 we may assume that 
 \begin{itemize}
\item[i.]  $\phi\in C(\C^2)$, using Tietze extension theorem, 
\item[ii.] $(0,0)\in \Gamma_1\times \{0\}\subset b\D_2$ 
	is a  non-trivial affine disc where $\Gamma_1=\{z\in \C:|z|<s_1\}$,
\item[iii.] $\D_2\subset \left\{(z_1,z_2)\in \C^2:|\text{arg}(z_2)|
	<\theta_1\right\}$ for some $0<\theta_1<\pi$,
\item[iv.] $H^{\D_1}_{\phi}R$ is compact on $A^2(\D_2)$.
\end{itemize}

Next we will use mollifiers (approximations to the identity) and 
trivial extensions to approximate $\phi$ on $z_1$-axis by suitable 
smooth functions $\phi_k$. We define 
$\widetilde{\phi}=\phi|_{\{(z_1,z_2)\in \C^2:z_2=0\}}$ and 
\[\phi_k=E(\widetilde{\phi}*\chi_k)\] 
where $*$ and $E$ denote the convolution and the 
trivial extension from $\{(z_1,z_2)\in \C^2:z_2=0\}$ to $\C^2$, 
respectively.  Then $\phi_k\to \phi$  uniformly on compact 
subsets in $\{(z_1,z_2)\in \C^2:z_2=0\}$ as $k\to \infty$ 
(see, for instance, \cite[2.29 Theorem]{AdamsFournierBook}).  
We note that, since $\phi_k$s are extended trivially in 
$z_2$-variable, the sequence $\{\phi_k\}$ is uniformly 
convergent on compact sets in $\C^2$. Hence, $\{\phi_k\}$ 
is uniformly bounded on $\Dc_1$. 
	
Lemma \ref{LemConv} implies that there exist 
$\delta>0, h\in C^{\infty}_0(\Gamma_1),$ and a 
subsequence $\{\phi_{k_j}\}$ such that  
\[\left|\langle (\phi_{k_j})_{\zb_1}, 
	h\rangle_{\Gamma_1}\right|\geq \delta>0\]
for all $j=1,2,3,\ldots$.  By passing to a subsequence, 
if necessary, we can assume that 
\[\left|\langle (\phi_k)_{\zb_1}, 
	h\rangle_{\Gamma_1}\right|\geq \delta>0\]
for all $k=1,2,3,\ldots$.

Since $\D_2$ has Lipschitz boundary  there exist $s_2>s_1$, 
$0<t_1<t_2$,  and $0<\theta_1<\pi/2<\theta_2<\pi$ such that  
\[\Gamma_1\times W_{t_1,\theta_1}
	\subset \D_1\subset  \D_2 
	\subset \Gamma_2 \times  W_{t_2,\theta_2} \]
 where $\Gamma_2 =\{z\in \C:|z|<s_2\}$ and  
\[W_{t_j,\theta_j} 
	=\left\{\rho e^{i\theta}\in \C:0<\rho<t_j, |\theta|
	<\theta_j\right\}\]
for $j=1,2$. 

We define a sequence of functions on $\D_2$ as 
\[f_j(z_1,z_2)= \frac{\alpha_j}{z_2^{\beta_j}}\]
where $\beta_j=1-1/j$ and $\alpha_j\to 0$ such that 
$\|f_j\|_{L^2(W_{t_1,\theta_1})}=1$ for all $j$. One can 
show that $\{f_j\}$ is a bounded sequence in $A^2(\D_2)$ as 
 $\|f_j\|_{L^2(W_{t_2,\theta_2})}$ are uniformly bounded. 
 Furthermore, the sequence  $\{f_j\}$ converges to zero 
 uniformly on compact subsets that are away from 
 $\{(z_1,z_2)\in \C^2:z_2=0\}$. Then $f_j\to 0$ weakly 
in $A^2(\D_2)$ as $j\to \infty$. Later on we will reach a 
contradiction by showing  that  $\|H^{\D_1}_{\phi}Rf_j\|_{\D_1}$ 
stays away from zero.

We remind the reader that  for any $f\in A^2(\D_2)$ and $k$ we have  
\[(H^{\D_1}_{\phi_k}Rf)_{\zb_1}
	=(Rf\phi_k)_{\zb_1}-(P_{\D_1}R(f\phi_k))_{\zb_1}
	 =R\left(f(\phi_k)_{\zb_1}\right).\]
Using the identity above (when we pass from second to third 
line below) and the Cauchy-Schwarz inequality (in $z_1$ on $\Gamma_1$ 
on the second inequality below) we get 
\begin{align*}
\delta^2= \delta^2 \|f_j\|^2_{L^2(W_{t_1,\theta_1})} 
& \leq  \int_{W_{t_1,\theta_1}} 
	\left|\langle (\phi_k)_{\zb_1}, h\rangle_{\Gamma_1}\right|^2
	f_j(.,z_2)\overline{f_j(.,z_2)} dV(z_2) \\
&=\int_{W_{t_1,\theta_1}}\langle (\phi_kf_j)_{\zb_1}, 
	h \rangle_{\Gamma_1} \overline{\langle (\phi_kf_j)_{\zb_1}, 
	h \rangle_{\Gamma_1}}dV(z_2)\\
&=\int_{W_{t_1,\theta_1}}
	\left|\langle (H^{\D_1}_{\phi_k}Rf_j)_{\zb_1},
	h \rangle_{\Gamma_1}\right|^2dV(z_2)\\
&=\int_{W_{t_1,\theta_1}}
	\left|\langle H^{\D_1}_{\phi_k}Rf_j,
	h_{z_1}\rangle_{\Gamma_1}\right|^2dV(z_2)\\
&\leq \int_{W_{t_1,\theta_1}}
	\|H^{\D_1}_{\phi_k}Rf_j(.,z_2)\|^2_{\Gamma_1} 
	\|h_{z_1}\|^2_{\Gamma_1}dV(z_2)\\
&= \|H^{\D_1}_{\phi_k}Rf_j\|^2_{\Gamma_1 \times W_{t_1,\theta_1}} 
	\|h_{z_1}\|^2_{\Gamma_1}.
\end{align*} 
Then,  for all $j$ and $k$, we have  
\begin{align*}
\frac{\delta}{\|h_{z_1}\|_{\Gamma_1}} 
	\leq \|H^{\D_1}_{\phi_k}Rf_j\|_{\D_1}.
\end{align*}
Using the facts that $\phi_k\to \phi$ uniformly on $\Gamma_1$, 
	the sequence $\{\phi_k\}$ is uniformly bounded on $\Dc_1$,  
and $f_j\to 0$ uniformly on compact subsets away from $z_1$-axis,  
one can show that   
\begin{align*}
\|H^{\D_1}_{\phi_k-\phi}Rf_j\|_{\D_1} 
	\leq \|(\phi_k-\phi)Rf_j\|_{\D_1} 
	\to 0 \text{ as } j,k\to \infty.
\end{align*}
Then we have 
\[\frac{\delta}{\|h_{z_1}\|_{\Gamma_1}} 
	\leq \|H^{\D_1}_{\phi_k}Rf_j\|_{\D_1} 
	\leq \|H^{\D_1}_{\phi_k-\phi}Rf_j\|_{\D_1} 
	+ \|H^{\D_1}_{\phi}Rf_j\|_{\D_1}.\]
Then if we let $j,k\to \infty$ we get 
\[0<\frac{\delta}{\|h_{z_1}\|_{\Gamma_1}} 
	\leq\liminf_{j\to \infty} \|H^{\D_1}_{\phi}Rf_j\|_{\D_1}.\] 
Finally, we conclude that $H^{\D_1}_{\phi}R$ is not compact on $A^2(\D_2)$ 
because if it were, the sequence $\{H^{\D_1}_{\phi}Rf_j\}$ would 
converge to zero in norm. Therefore, using Lemma \ref{LemLocalization}, 
we reach a contradiction with the assumption that $H_{\phi}$ is compact.
\end{proof}

%%%%%%%%%%%%%%%%%%%%%%%
%%%%%%%%%%%%%%%%%%%%%%%
\section*{Proof of Theorem \ref{ThmConvex} and Corollaries}

In Lemma \ref{LemNotHolo} below we will use the following 
notation: $L_{z_0,z_1}:\mathbb{D}\to b\D$ is defined as 
$L_{z_0,z_1}(\xi)=z_0+\xi z_1$ where $z_0,z_1\in \C^n$.
\begin{lemma}\label{LemNotHolo}
Let $\D$ be a bounded convex domain in $\C^n$ and 
$\phi\in C(\Dc)$. Assume that there exists a holomorphic 
function $f:\mathbb{D}\rightarrow b\D$ so that $\phi\circ f$ 
is not holomorphic.  Then there exist $z_0\in b\D, z_1\in \C^n$ 
such that $L_{z_0,z_1}(\mathbb{D})\subset b\D$  and 
$\phi\circ L_{z_0,z_1}$ is not holomorphic.  
\end{lemma}

\begin{proof}
We first use \cite[Lemma 2]{CuckovicSahutoglu09} 
(see also \cite[Section 2]{FuStraube98}) to conclude that  
the convex hull of $f(\mathbb{D})$ is contained in an 
affine variety  $V\subset b\D$.  So $\phi|_V$ is not 
holomorphic. Next we use the following fact: 
a continuous function is holomorphic on an open set 
$U$ if and only if it is holomorphic on every complex 
line in $U$.  Therefore, we conclude that there is a 
complex line $L_{z_0,z_1}(\mathbb{D})\subset V$ such 
that $\phi\circ L_{z_0,z_1}$ is not holomorphic on $\mathbb{D}$.   
\end{proof}

We will need the following lemma in the proof of 
Theorem \ref{ThmConvex}. 
\begin{lemma}\label{LemFunc}
Let $\D$ be a domain in $\C^n$ with Lipschitz boundary 
such that $0\in b\D$. Then the function $f(z)=|z_{n}|^{-p}$ 
is not square integrable on $\D$ for $p\geq n$.    
\end{lemma}
\begin{proof}
We can use rotation to assume that positive $y_n$-axis 
is transversal to $b\D$ and  there exists $\alpha,\ep>0$ 
such that  
\[W_{\ep,\alpha}
	=\{(z',z_n)\in \C^{n-1}\times \C :|z'|^2+x_n^2<\alpha^2y_n^2, -\ep<y_n<0\} 
	\subset \D\] 
where $z_n=x_n+iy_n$. In the following calculation 
$w_{\ep,\alpha}=\{x_n+iy_n\in \C:|x_n|+\alpha y_n<0,-\ep<y_n<0\}$ 
is a wedge in $z_n$-axis. 
\begin{align*}
\int_{\D}|z_n|^{-2p}dV(z)
\geq & \int_{W_{\ep,\alpha}}|z_n|^{-2p}dV(z',z_n)\\
=&\int_{z_n\in w_{\ep,\alpha}}
	\int_{|z'|^2<\alpha^2y_n^2-x_n^2}|z_n|^{-2p} dV(z')dV(z_n) \\ 
\gtrsim & \int_{z_n\in w_{\ep,\alpha}}
	(\alpha^2y_n^2-x_n^2)^{n-1}|z_n|^{-2p} dV(z_n)\\
\gtrsim &  \int_0^{\ep} \frac{1}{r^{1+2(p-n)}}dr.
\end{align*}
Therefore, if $p\geq n$ the function $f(z)=|z_{n}|^{-p}$ is not 
square integrable on $\D$ as the last integral above is infinite.  
\end{proof}

\begin{proof}[Proof of Theorem \ref{ThmConvex}]
Using holomorphic linear translation, if necessary, we may assume 
that $\D\subset \{y_n<0\}$ and the origin is in the boundary of $\D$. 
Furthermore, by Lemma \ref{LemNotHolo} we may assume 
that $0\in \Gamma=\{z\in \C:(z,0,\ldots,0)\in b\D\}$ is a non-trivial 
affine analytic disc such that $\phi(.,0,\ldots, 0)$ is not holomorphic. 
Finally, since convex domains have Lipschitz boundary (see, 
for instance, \cite{WayneState}), we may also assume that 
positive $y_n$-axis is transversal to $b\D$ on $\Gamma$. 

Let $\D^{z_1}=\{z''\in \C^{n-1}:(z_1,z'')\in \D\}$ be the slice of $\D$ 
perpendicular to $\Gamma$ at $z_1\in \Gamma$.  Convexity of $\D$ 
and the fact that $0\in \Gamma \times \{0\}\subset b\D$ imply that 
\[\left(\frac{z_1}{2},\frac{z''}{2}\right) 
=\frac{1}{2}(z_1,0)+\frac{1}{2}(0,z'')\in \D\] 
for $z_1\in \Gamma$ and $z''\in \D^0$. That is, $\D^0\subset 2\D^{z_1/2}$ 
for $z_1\in \Gamma$. Equivalently, $\D^0\subset 2\D^{z_1}$ 
for $z_1\in \frac{1}{2}\Gamma$. Hence, 
\begin{align}\label{Eqn1}
\frac{1}{2}(\Gamma \times \D^0) \subset \D.
\end{align}

To get another inclusion, let $0<r_1$ such that 
$\{z_1\in \C:|z_1|<r_1\}\subset \Gamma$ and $z''\in \D^{z_1}$. 
Then, we have $(z_1,z'')\in \D$ and  $(-z_1,0)\in \Gamma$. Hence  
\[\left(0,\frac{z''}{2}\right) 
=\frac{1}{2}(-z_1,0) +\frac{1}{2}(z_1,z'')\in \D.\] 
That is, $\frac{1}{2}\D^{z_1}\subset \D^0$ for $|z_1|<r_1$. Namely,  
$\D^{z_1}\subset 2\D^0$ for $|z_1|<r_1$. Hence, 
\[\D\cap B(0,r_1)\subset 2 (\Gamma \times \D^0).\]
Therefore, combining the previous inclusion with \eqref{Eqn1} we get  
\[\frac{1}{2}(\Gamma \times \D^0)\cap B(0,r_1)
	\subset \D\cap B(0,r_1) 
	\subset 2 (\Gamma \times \D^0).\]
 
Next we will use Lemma \ref{LemFunc} to produce a bounded sequence 
$\{f_j\}$ in $A^2(\D)$ that is convergent to zero weakly but its image 
under a ``local" Hankel operator does not converge to zero. 

Since, by Lemma \ref{LemFunc}, the function $f(z)=z_n^{-n+1}$ is not 
square integrable on $\D^0$ (an $(n-1)$-dimensional slice of $\D$) 
and the $L^2$-norm of $(z_n-i\delta)^{-n+1}$ on $\D^0$ continuously 
depends on $\delta>0$, we can choose a positive sequence $\{\delta_j\}$ 
such that $\delta_j\to 0$ as $j\to\infty$ and $\|f_j\|_{\frac{1}{2}\D^0}=1$ 
where     
\begin{align}\label{eq1}
	f_j(z)=\frac{1}{j(z_n-i\delta_j)^{n-1}}.
\end{align}
Furthermore, $|f_j(4z)|\leq |f_j(z)|$ for all 
$z\in \frac{1}{2}\D^0$ as $\D\subset \{y_n<0\}$ and $\delta_j>0$. 
Then,  
\[\int_{2\D^0}|f_j(\xi)|^2dV(\xi) 
=16^{n-1}\int_{\frac{1}{2}\D^0}|f_j(4\eta)|^2dV(\eta) 
\leq  16^{n-1}\int_{\frac{1}{2}\D^0}|f_j(\eta)|^2dV(\eta)=16^{n-1}.\] 
Hence $\{f_j\}$ is a bounded sequence in $A^2(\D)$ 
(as $\|f_j\|_{2\D^0}$ is uniformly bounded) and  
$f_j\to 0$ weakly in $A^2(\D)$ as $j\to \infty$.

The rest of the proof follows the proof of Theorem \ref{ThmC2}. 
Namely, we define $\Gamma_1=\{z\in \C:|z|<\frac{r_1}{2}\}$ and 
$\widetilde{\phi}=\phi|_{\Gamma\times \{0\}}$. Without 
loss of generality, we may assume that $\phi\in C(\C^n)$. 
We define  
\[\phi_k=E(\widetilde{\phi}*\chi_k)\] 
where $E$ denotes trivial extension from $\{(z_1,z'')\in \C^2:z''=0\}$ 
to $\C^n$, respectively. Using Lemma \ref{LemConv} we can choose  
$\delta>0$ and $h\in C^{\infty}_0(\Gamma_1)$ so that, by passing 
to a subsequence if necessary, we can assume that 
\[\left|\langle (\phi_k)_{\zb_1}, h\rangle_{\Gamma_1}\right|
	\geq \delta>0\]
for all $k=1,2,3,\ldots$. Then for $\D_1=\D\cap B(p,r_1)$ we get 
\[\frac{\delta}{\|h_{z_1}\|_{\Gamma_1}} 
	\leq \|H^{\D_1}_{\phi_k}Rf_j\|_{\D_1}\]
for all $j,k$ where $R:A^2(\D)\to A^2(\D_1)$ is the restriction 
operator. Then letting $j,k\to \infty$ we get 
\begin{align}\label{eq2}
	0<\frac{\delta}{\|h_{z_1}\|_{\Gamma_1}} 
	\leq\liminf_{j\to \infty} \|H^{\D_1}_{\phi}Rf_j\|_{\D_1}.
\end{align} 
Hence,  $H^{\D_1}_{\phi}R$ is not compact and we reach a 
contradiction with the assumption that $H_{\phi}$ is compact. 
Therefore,  the proof of Theorem \ref{ThmConvex} is complete.   
\end{proof}

\begin{proof}[Proof of Corollary \ref{Cor1}]
Suppose $\D\subset \mathbb{C}^n$ is a bounded locally convexifiable 
domain, $\phi\in C(\overline{\D})$ is such that $H_{\phi}$ is compact 
on $A^2(\D)$, and $f:\mathbb{D}\rightarrow b\D$ is a holomorphic 
function.  Let $p\in f(\mathbb{D})$ and choose $r>0$ such that 
$B(p,r)\cap \D$ is convexifiable. Furthermore, without loss of generality, 
we may assume that the range of $f$ is contained in $B(p,r/2)$. Then 
using Lemma \ref{LemLocalization} and Lemma \ref{LemHoloInvariance} 
(and shrinking $r$ is necessary) we may assume that  $U=B(p,r)\cap \D$ 
is convex and $H^V_{\phi}R$ is compact on $A^2(U)$ where 
$V=B(p,r/2)\cap \D$ and $R:A^2(U)\to A^2(V)$ is the restriction 
from $U$ onto $V$. Then the proof of Theorem \ref{ThmConvex} 
implies that $\phi\circ f$ is holomorphic. 
\end{proof}

\begin{proof}[Proof of Corollary \ref{Cor2}]
Let $\D\subset \C^2$ be a bounded pseudoconvex complete Reinhardt 
domain, $\phi\in C(\Dc)$, and $H_{\phi}$ is compact on $A^2(\D)$. 
By \cite[Theorem 3.28]{RangeBook}, $\D$ is locally convexifiable away from 
the coordinate axes under the map $(z_1,z_2)\to (\log z_1,\log z_2)$. 

Assume that there is a non-trivial analytic disc in the boundary away 
from the coordinate axes. Then there exists a non-constant holomorphic 
function $f:\mathbb{D}\rightarrow b\D$  such that 
$f(\mathbb{D}) \subset  \{(z_1,z_2)\in \C^2:z_1\neq 0 \text{ and } z_2\neq 0\}$.
Using an argument similar to the one in the proof of 
Corollary \ref{Cor1} we conclude that $\phi\circ f$ is holomorphic. Therefore, 
$\phi$ is holomorphic along any disc away from the coordinate axis. 

Next, if the disc intersects one of the coordinate axis,  without loss of generality, 
we assume that $f(\mathbb{D}) \cap \{(z_1,z_2)\in \C^2:z_1= 0\}\neq \emptyset$. 
Let $f=(f_1,f_2)$.  Then $f_1:\mathbb{D}\to \C$ has a zero. Since zeroes of a 
holomorphic function on a planar domain are isolated, we can choose $f$ 
so that $f_1(z)=0$ if and only if $z=0$. Therefore, we may assume that  
$f(z)$ is on a coordinate axis if and only if $z=0$. Then,  similarly as 
in the previous paragraph, we conclude that $\phi\circ f$ is holomorphic 
on $\mathbb{D}\setminus \{0\}$. Furthermore, 0 is a removable singularity 
for $\phi\circ f$ as $\phi\circ f$ is continuous on $\mathbb{D}$. 
That is, $\phi\circ f$ is holomorphic on $\mathbb{D}$.
\end{proof}	

\section*{Acknowledgment}

We would like to thank Emil Straube for reading an earlier manuscript of this 
paper and for providing us with valuable comments. We also thank the referee 
for feedback that has improved the exposition of the paper.

%%%%%%%%%%%%%%%%%%%%%%%%%%%%%%%%%%%%
%%%%%%%%%%%%%%%%%%%%%%%%%%%%%%%%%%%% 
%\bibliographystyle{amsalpha}
%\bibliography{CompHankel}

\end{document}